\documentclass{amsart}

\usepackage{amssymb}
\usepackage{array}
\usepackage[all,cmtip]{xy}

\title{Galois descent for real spectra}
\author{Romie Banerjee}
\date{\today}
\address{Department of Mathematics, Indian Institute of Science Education and Research, Bhopal, India}
\email{romie@iiserb.ac.in}

\begin{document}

\begin{abstract}
We prove analogs of faithfully flat descent and Galois descent for categories of modules over $E_{\infty}$-ring spectra using the $\infty$-categorical Barr-Beck theorem proved by Lurie. In particular, faithful $G$-Galois extensions are shown to be of effective descent for modules. Using this we study the category of $ER(n)$-modules, where $ER(n)$ is the $\mathbb{Z}/2$-fixed points under complex conjugation of a generalized Johnson-Wilson spectrum $E(n)$. In particular, we show that $ER(n)$-modules is equivalent to $\mathbb{Z}$/2-equivariant $E(n)$-modules as stable $\infty$-categories. 
\end{abstract}

\maketitle
\tableofcontents

\theoremstyle{definition}
\newtheorem{definition}{Definition}[section]
\newtheorem*{Proofformal}{Proof of Theorem \ref{formal}}
\newtheorem{example}{Example}[section]
\newtheorem{remark}{Remark}[section]

\theoremstyle{plain}
\newtheorem{prop}{Proposition}[section]
\newtheorem{lemma}{Lemma}[section]
\newtheorem{theorem}{Theorem}[section]

\newtheorem{cor}{Corollary}[section]

\section{Introduction}

Let $K\subset L$ be a finite Galois extension of fields with Galois group $\hbox{Gal}(L|K)=G$. Then the map $K\to L$ is of effective descent for modules. This means any $K$-module is equivalent to an $L$-module along with some descent data. More precisely the map 
$$\xymatrix{
K\hbox{- vector spaces} \ar[rr]^{-\otimes_KL} &&G-\left(L\hbox{- vector spaces}\right)\\
}$$
is an equivalence of categories. Here $G-\left(L\hbox{- vector spaces}\right)$ is the category of $L$- vector spaces $N$ with a $G$-action which is semilinear in the sense that $\sigma(cx) = \sigma(c)\sigma(x)$ for all $\sigma \in G$, $c\in L$ and $x\in N$. This is classical Galois descent for fields. A similar descent formalism exists for Galois extensions of commutative rings as well.

Rognes (in \cite{Rognes}) has defined Galois extensions in the category of $E_{\infty}$-ring spectra. In the first part of this paper we will prove Galois descent for {\it faithful} Galois extensions of ring spectra. In particular we show that if $A\to B$ is a faithful $G$-Galois extension of $E_{\infty}$-ring spectra and $G$ is a finite group then the map
$$\xymatrix{
A\hbox{-mod} \ar[rr]^{-\wedge_AB} &&G-\left(B\hbox{-mod}\right)\\
}$$ is an equivalence of stable $\infty$-categories. Here $G-\left(B\hbox{-mod}\right)$ is the category of $G$-equivariant $B$-modules (Def.\ref{semilinear}). The main tool that we use is the comonadic formulation for descent in $\infty$-categories and the Barr-Beck-Lurie criteria for comonadicity. Along the way we also prove an analog of faithfully flat descent for $E_{\infty}$-rings.

Some examples of global faithful Galois extensions of $E_{\infty}$-rings are the following: 

\begin{itemize}
\item $HR \to HS$ where $R\to S$ is a Galois extension of commutative rings. The Galois group is $\hbox{Gal}(S|R)$. (\cite[Lemma 4.2.5]{Rognes})

\item $KO\to KU$ with Galois group $\mathbb{Z}/2$ (\cite[Prop.5.3.1]{Rognes})

\item $TMF[1/n] \to TMF(n)$ with Galois group $GL_2(\mathbb{Z}/n\mathbb{Z})$(\cite{MM})

\end{itemize}

In the second part of this paper we show that Real Johnson-Wilson theories (at the prime $2$) produce examples of faithful $\mathbb{Z}/2$-Galois extensions. Complex conjugation acts on the complex $K$-theory spectrum $KU$ and the homotopy fixed points of this action is $KO$. The complex orientation $MU \rightarrow KU$ is equivariant with respect to this $\mathbb{Z}/2$-action. Hu and Kriz (\cite{HK}) have realized this as a map of genuine $\mathbb{Z}/2$-equivariant spectra $MU_{\mathbb{R}} \rightarrow K_{\mathbb{R}}$ where $K_{\mathbb{R}}$ is Atiyah's {\it real} $K$-theory (\cite{Atiyah}). More generally, let $X$ be any $2(2^n-1)$-periodic {\em generalized} Johnson-Wilson theory (see section \ref{generalized}) with coefficient ring $\mathbb{Z}_{(2)}[u_1,\ldots,u_n^{\pm1}]$ with $|u_i|=2(2^i-1)$ and orientation $MU \to X$ with kernel $\langle u_i, i>n \rangle$. Then Hu and Kriz constructs an equivariant refinement of this orientation to a map $$MU_{\mathbb{R}} \to X_{\mathbb{R}}.$$ We shall call $X_{\mathbb{R}}$ a {\em Real generalized Johnson-Wilson spectrum} and its $\mathbb{Z}/2$ homotopy fixed points $XR$ a {\em real generalized Johnson-Wilson spectrum}. 

We show the existence of a fibration (Thm.\ref{KWfib})
\begin{equation}\label{fibration}
\xymatrix{
\Sigma^{\lambda(n)}XR \ar[r]^{x(n)} &XR \ar[r] &X\\
}
\end{equation}
where the element $x(n)$ has degree $\lambda(n) = 2^{2n+1} -2^{n+2}+1$ and is nilpotent with $x(n)^{2^{n+1}-1}=0$. This is analogous to the Kitchloo-Wilson fibration \cite[Thm.1.7]{KW1} relating the classical Johnson-Wilson theories $E(n)$ to their real counterparts. Using this fibration we can add to our list of faithful Galois extensions: 
\begin{itemize}
\item $XR \to X$ with Galois group $\mathbb{Z}/2$ (Thm.\ref{faith}) when $X$ is a generalized Johnson-Wilson theory admitting a $E_{\infty}$-ring structure:
\end{itemize}

The main results of this paper can be summarized as follows:

\begin{theorem}
\hspace{3mm}

\begin{enumerate}

\item Let $f:A\to B$ be a map of $E_{\infty}$-rings so that $B$ is faithful and dualizable over $A$. Then the map $f$ is of effective descent for modules (Section \ref{fdd}, Thm.\ref{amit}). In particular, there is an equivalence of stable $\infty$-categories $$A\hbox{-mod} \simeq \varprojlim (B^{\wedge_A^n}\hbox{-mod}).$$

\item Any faithful $G$-Galois extension $A \to B$ of $E_{\infty}$-rings ($G$ finite) is of effective descent for modules. In particular, there is an equivalence of stable $\infty$-categories (Section \ref{galois}, Thm.\ref{GaloisDesc}) $$A\hbox{-mod} \simeq (B\hbox{-mod})^{hG}.$$ 

\end{enumerate}
\end{theorem}

\begin{theorem}
 Let $X$ be a generalized Johnson-Wilson spectrum admitting an $E_{\infty}$-ring structure and $X_{\mathbb{R}}$ the Real spectrum associated with it. Let $XR$ denote the homotopy fixed points spectrum. Then the canonical map $XR\to X$ is a faithful $\mathbb{Z}/2$-Galois extension (Thm. \ref{faith}). As a consequence, there is an equivalence of stable $\infty$-categories $$XR\hbox{-mod} \simeq (X\hbox{-mod})^{h\mathbb{Z}/2}.$$

\end{theorem}

\begin{remark}
A similar Galois descent result has appeared in Meier's thesis (\cite[Prop.6.2.6]{Meier}). The proof makes use of the main result of \cite{SS}.

The only examples of generalized Johnson-Wilson spectra ($p=2$) admitting $E_{\infty}$ structures that are known of are $KU_{(2)}$ for $n=1$ and ${TMF_1(3)}_{(2)}$ for $n=2$ (\cite[Thm.1.1 and Prop.8.3]{LN}). However, the $K(n)$ localized Johnson-Wilson spectra $E(n)^{\wedge}_{I_n} = L_{K(n)}E(n)$ all admit essentially unique $E_{\infty}$-ring structures, by work of Baker-Richter (\cite{BR}). 

\end{remark}

\subsection{Organization} The paper is organized into two parts. 

In \ref{classdesc} we review the classical descent theory for a map of commutative rings. In \ref{stable} we recall the essential properties of small stable $\infty$-categories and in \ref{desc} we define what it means for a map of $E_{\infty}$-rings to be of effective descent for modules, using the language of $\infty$-categories. In \ref{comonads} we provide definitions of $\infty$-comonads and their $\infty$-category of comodules and state the Barr-Beck-Lurie criteria for comonadicity. In \ref{hdd} we show that the $\infty$-category of descent data associated to a map of $E_{\infty}$-rings is equivalent to a category of comodules over an $\infty$-comonad (Thm.\ref{hdd=comod}). In \ref{fdd} we apply the Barr-Beck-Lurie criteria to show that a map $A\to B$ of $E_{\infty}$-rings is of effective descent for modules if the extension is faithful and $B$ is dualizable over $A$. In \ref{galois} we show that if furthermore the map is Galois then category of descent data comes from a $G$-equivariant structure (Thm.\ref{GaloisDesc}). 

In sections \ref{Z2} and \ref{real} we recall some facts about $\mathbb{Z}/2$-equivariant spectra and the construction of real-oriented spectra. In section \ref{generalized} we define real versions of generalized Johnson-Wilson spectra. In section \ref{fib} we construct the fibration (\ref{fibration}) and finally we show (in Thm.\ref{faith}) that $XR \to X$ is a faithful Galois extension when $X$ is a generalized Johnson-Wilson theory admitting a $E_{\infty}$-ring structure. The proof depends on the computation of  homotopy of $X_{\mathbb{R}}$ using the Borel spectral sequence. This is done in the appendix \ref{comps}.

\subsection{Acknowledgements} I have benefited greatly from conversations with Nitu Kitchloo, Igor Kriz, Lennart Meier, Jack Morava and Andrew Salch. I am also indebted to Mark Behrens and an anonymous referee for making valuable comments in the earlier drafts of the paper.

\section{Descent for modules}

\subsection{Classical descent theory}\label{classdesc}

Let $f:A\to B$ be a map of commutative rings. The question of descent is the following: given a $B$-module $N$, what data on $N$ determines an $A$-module $M$, together with an isomorphism of $B$-modules $M\otimes_AB \simeq N$?

Given a $B$-module $N$ the descent data for $N$ is necessary gluing data for the module to descend to a module over $A$. The descent data can be defined precisely in the following way. Let us begin with the following  truncated semi-cosimplicial diagram.

\begin{equation}\label{CB}
\xymatrix{
B \ar@<1ex>[r]^{\phi_1} \ar@<-1ex>[r]^{\phi_2} &B\otimes_AB \ar[r]^{\phi_{12}} \ar@<2ex>[r]^{\phi_{23}} \ar@<-2ex>[r]^{\phi_{13}} &B\otimes_AB\otimes_AB\
}
\end{equation}

The maps are $\phi_1= B\otimes 1, \phi_2 = 1\otimes  B, \phi_{12} = B\otimes B\otimes 1, \phi_{23} = 1\otimes B \otimes B$ and $\phi_{13} = B \otimes 1 \otimes B$. The diagram is semi-cosimplicial in the sense the following equalities hold: 

\begin{equation}
\begin{split}
\phi_{12} \phi_1 = \phi_{13}\phi_1\\
\phi_{12} \phi_2 = \phi_{23} \phi_1\\
\phi_{13} \phi_2 = \phi_{23} \phi_2\\
\end{split}
\end{equation}

A descent datum for a $B$ module $N$ is the following:

\begin{itemize}
\item an isomorphism $\phi: \phi_1^*N \simeq \phi_2^*N$ of $B\otimes_AB$-modules,
\item satisfying a cocycle condition: $\phi_{13}^*\phi = \phi_{23}^*\phi \circ \phi_{12}^*\phi$ in the category of $B\otimes_AB\otimes_AB$-modules.
\end{itemize}

The $B$-modules together with descent data form a category $Desc(f)$. If $M$ is an $A$-module and $N = M\otimes_AB$ then there is an obvious descent datum from the isomorphisms $\phi:\phi_1^*N = B\otimes_AN \simeq B\otimes_AB\otimes_AM \simeq N \otimes_AB = \phi_2^*N$. 
There is a map $$\hbox{Mod}(A) \to Desc(f)$$ given by $M \mapsto (M\otimes_AB,\phi)$. The map $f:A\to B$ is said to have {\em effective descent} for modules when this is an equivalence of categories. 

The category of descent data can be described as the limit of a truncated semi-cosimplicial category (induced by (\ref{CB})) in the $2$-category of categories.

$$Desc(f) = \varprojlim\left( 
\xymatrix{\hbox{Mod}(B) \ar@<1ex>[r] \ar@<-1ex>[r] &\hbox{Mod}(B\otimes_AB) \ar[r] \ar@<2ex>[r] \ar@<-2ex>[r] &\hbox{Mod}(B\otimes_AB\otimes_AB) 
}
\right)
$$

There is another description of the category of descent data as the category of comodules over a {\em comonad} on $\hbox{Mod}(B)$. Recall (see \cite{MacLane}) that a comonad $K$ on a category $\mathcal{D}$ is a coalgebra object in the functor category $\hbox{End}(\mathcal{C})$ with respect to the composition monoidal structure. Let $\mathcal{D}_K$ denote the category of $K$-comodules in $\mathcal{D}$. The Eilenberg-Moore adjunction $(U_K\dashv F_K):\mathcal{D}_K \to \mathcal{D}$ is the final representation of $K$; for any other representation of $K$ via a pair of adjoint functors $(F \dashv G): \mathcal{C} \rightleftarrows \mathcal{D}$ (so that $K = F\circ G$), there is a canonical map $\mathcal{C} \to \mathcal{D}_K$. The functor $F$ is said to be {\em comonadic} if this is an equivalence of categories. The Barr-Beck theorem gives necessary and sufficient conditions for the comonadicity of $F$ (\cite[VI.7]{MacLane}).

The map $f:A \to B$ induces a pair of adjoint functors $$(f^* \dashv f_*):\hbox{Mod}(A) \rightleftarrows \hbox{Mod}(B).$$ Here $f_* = -\otimes_AB$ and $f_*$ is the forgetful function. Let $\hbox{Mod}(B)_K$ be the Eilenberg-Moore category of comodules over the comonad $K=f_*f^*$ on $\hbox{Mod}(B)$. There is a forgetful functor $U:Desc(f) \to \hbox{Mod}(B)$ which forgets the descent data. The functor $U$ admits a {\em right} adjoint $F$, $F(M) = (f_*f^*(M), \phi)$ and $F\circ U = K$. So $$(U \dashv F): Desc(f) \rightleftarrows \hbox{Mod}(B)$$ is a presentation of the comonad $K$. Therefore by the universal property of the Eilenberg-Moore representation, there is a canonical map $ Desc(f) \to \hbox{Mod}(B)_K$.

\begin{prop}\label{comonadicdescent} The canonical map $Desc(f) \to \hbox{Mod}(B)_K$ is an equivalence of categories, i.e. $U$ is comonadic.
\end{prop}

\begin{proof}
From the square below

$$\xymatrix{
\hbox{Mod}(A) \ar@<1ex>[rr]^{f^*} \ar@<1ex>[d]^{f^*} &&\hbox{Mod}(B) \ar@<1ex>[ll]^{f_*} \ar@<1ex>[d]^{\phi_1^*, \phi_2^*} \\
\hbox{Mod}(B) \ar@<1ex>[rr]^{\phi_1^*, \phi_2^*} \ar@<1ex>[u]^{f_*} &&\hbox{Mod}(B\otimes_A B) \ar@<1ex>[ll]^{{\phi_1}_*, {\phi_2}_*} \ar@<1ex>[u]^{{\phi_1}_*, {\phi_2}_*}
}$$

we have the following isomorphisms of comonads on $\hbox{Mod}(B)$: 

\begin{equation}\label{BC1} K = f_*f^* \simeq {\phi_1}_*\phi_2^* \simeq {\phi_2}_*\phi_1^* \end{equation}

A $K$-comodule structure on $N$ is a given by a $B$-map $N \to N\otimes_AB = f_*f^*N \simeq {\phi_1}_*\phi_2^*N \simeq {\phi_2}_*\phi_1^*N$. So by adjunction, the comodule structure map $a$ gives the descent datum $\phi:\phi_1^*N \to \phi_2^*N$. 
\end{proof}

The question of effective descent therefore is equivalent to the question of whether the map $$f^*:\hbox{Mod}(A) \to \hbox{Mod}(B)$$ is comonadic, which may be verified by applying the theorem of Barr and Beck. 

\begin{theorem} (Grothendieck) Let $f:A\to B$ be faithfully flat, then $f^*$ is comonadic, hence $f$ is of effective descent for modules. 
\end{theorem}

Our aim is to obtain a homotopical version of this theorem, where we can replace the abelian categories of modules with their derived categories. However it is known that one cannot glue objects in the derived category of a cover to obtain objects in the derived category of the base. What one can do instead is consider the more enriched stable $\infty$-category or dg-category versions of the derived categories. In this paper we will work with stable $\infty$-categories as they contain important examples from stable homotopy theory that we want our results to apply to. 

\subsection{Stable $\infty$-categories}\label{stable}

By an $\infty$-category we shall mean an $(\infty,1)$-category modelled using a {\em weak-Kan complex}. This theory was first developed by Boardman-Vogt and Joyal and later in great detail by Lurie in \cite{HTT} and \cite{HA}. These will be our main references.

\begin{definition}(\cite[Def.1.1.2.4]{HTT})
A simplicial set $K$ is an {\em $\infty$-category} if it satisfies the following condition: for any $0<i<n$, any map $f_0:\Lambda^n_i \to K$ admits (possibly non-unique) extension $f:\Delta^n \to K$. Here $\Lambda_i^n \subseteq \Delta$ denotes the $i$-th horn, obtained from the simplex $\Delta^n$ by deleting the face opposite the $i$-th vertex.
\end{definition}

Let $K$ be a simplicial set underlying an $\infty$-category $\mathcal{C}$. The objects of $\mathcal{C}$ are the elements of $K_0$, the morphisms of $\mathcal{C}$ are the elements of $K_1$. The hom set $\hbox{Maps}_{\mathcal{C}}(x,y)$ is a Kan complex. So every $\infty$-category has an underlying simplicial category.

A {\em functor} between $\infty$-categories is a map of simplicial sets. The functors betwen $\infty$-categories $\mathcal{C}$ and $\mathcal{D}$ assemble in an $\infty$-category $\hbox{Fun}(\mathcal{C},\mathcal{D})$. We say a functor is an {\em equivalence of $\infty$-categories} when the map of the underlying simplicial categories is a Dwyer-Kan equivalence. The {\em homotopy category} of $\mathcal{C}$ is the homotopy category of the underlying simplicial category.

\begin{definition}(\cite[Chapter 3]{HTT})
Let $\mathcal{C}\hbox{at}^{\Delta}_{\infty}$ be the simplicial category whose objects are small $\infty$-categories. Given two $\infty$-categories $\mathcal{C}$ and $\mathcal{D}$ define the mapping space $\hbox{Maps}_{\mathcal{C}\hbox{at}^{\Delta}_{\infty}}(\mathcal{C},\mathcal{D})$ to be the maximal Kan complex contained in the $\infty$-category of functors $\hbox{Fun}(\mathcal{C},\mathcal{D})$. The $\infty$-category $\mathcal{C}\hbox{at}_{\infty}$ is defined to be the simplicial nerve $N(\mathcal{C}\hbox{at}^{\Delta}_{\infty})$.
\end{definition} 

The $\infty$-category $\mathcal{C}\hbox{at}_{\infty}$ admits small limits (\cite[Section 3.3.3]{HTT}).

\begin{definition}(\cite[Def.1.1.1.9]{HA})
An $\infty$-category $\mathcal{C}$ is {\em stable} if 
\begin{enumerate}
\item There is a zero object $0 \in \mathcal{C}$
\item Every morphism in $\mathcal{C}$ has a fiber and a cofiber
\item A triangle in $\mathcal{C}$ is a fiber sequence if and only it is a cofiber sequence
\end{enumerate}
\end{definition}

The homotopy category of a stable $\infty$-category is canonically triangulated (\cite[1.1.2]{HA}). So stable $\infty$-categories can be thought of as natural enrichments of triangulated categories. 

Some examples of stable $\infty$-categories:

\begin{itemize}
\item For $R$ an ordinary commutative ring, the $\infty$-category of unbounded chain complexes of modules over $R$ is stable \cite[1.3.5]{HA}, called the derived category $\mathcal{D}(R)$. Its homotopy category is the derived category $D(R)$ of the ring $R$. 

\item For $X$ a scheme or an algebraic stack one can assign a stable $\infty$-category $QC(X)$ which has the usual unbounded quasicoherent derived category $D_{qc}(X)$ as the homotopy category (see \cite{integral}). 

\item For a $E_{\infty}$-ring $A$, the $\infty$-category of $A$-modules $A$-mod is stable \cite[1.4]{HA}. When $A=S^0$, the homotopy category is the classical stable homotopy category. When $A=HR$ for a discrete ring $R$, the homotopy category is the derived category $D(R)$ \cite[Thm. 4.2.4]{EKMM}. 
 
\end{itemize}
 
There is a good notion of homotopy limits in the $\infty$-category of stable $\infty$-categories. This is important for the descent formalism we develop later. We make precise statements here.

Given two stable $\infty$-categories $\mathcal{C}$ and $\mathcal{D}$, an exact functor between them is an $\infty$ functor that preserves $0$ and fiber sequences. The identity functor is exact and composition of exacts functors is exact. This gives us the following definition.

\begin{definition}
The $\infty$-category $\mathcal{C}\hbox{at}^{Ex}_{\infty}$ is the subcategory of $\mathcal{C}\hbox{at}_{\infty}$ whose objects are small stable $\infty$-categories and morphisms are exact functors.
\end{definition}

\begin{theorem}(\cite[Thm. 1.1.4.4]{HA}) The $\infty$-category $\mathcal{C}\hbox{at}_{\infty}^{Ex}$ admits small limits and $\mathcal{C}\hbox{at}^{Ex}_{\infty}\subseteq \mathcal{C}\hbox{at}_{\infty}$ preserves small limits.
\end{theorem}

\subsection{Higher descent theory}\label{desc}

Let $(\hbox{Sp},\wedge)$ denote the symmetric monoidal stable $\infty$-category of spectra. Denote by $\hbox{CAlg}(\hbox{Sp})$ the category of $E_{\infty}$-rings. If $A$ is an $E_{\infty}$ ring, the category $A\hbox{-mod}$ is symmetric monoidal by the relative smash product $\wedge_A$. Denote by $\hbox{CAlg}(A\hbox{-mod})$ the category of commutative $A$-algebras.

Given a map of $f:A\to B$ of $E_{\infty}$-rings there is a map of stable $\infty$-categories $$f^*:A\hbox{-mod} \to B\hbox{-mod}$$ which is defined on the $0$-simplices as $f^*(M) = M\wedge_AB$. This extends (see \cite{HA}) to a $\infty$-functor $$\hbox{QC}: \hbox{CAlg}(\hbox{Sp}) \to \mathcal{C}\hbox{at}^{Ex}_{\infty}.$$

Let $A^{\bullet}:N(\Delta) \to \hbox{CAlg}(\hbox{Sp})$ be a cosimplicial object in $E_{\infty}$-rings. The functor $\hbox{QC}$ applied to $A^{\bullet}$ levelwise produces a cosimplicial stable $\infty$-category $\hbox{QC}(A^{\bullet})$. Since the category $\mathcal{C}\hbox{at}^{Ex}_{\infty}$ closed under small limits, the totalization $\hbox{Tot}(\hbox{QC}(A^{\bullet}))$ is a stable $\infty$-category.

\begin{definition}(\cite[Def. 8.2.1]{Rognes})
Let $f: A \to B$ be a map of $E_{\infty}$-rings. The {\em Amitsur complex} associated with $f$ is a cosimplicial commutative $A$-algebra, $$C^{\bullet}(B/A):N(\Delta) \to \hbox{CAlg}(A\hbox{-mod})$$ with $C^{q}(B/A)=B^{\wedge^{q+1}_A}$, coaugmented by $A \rightarrow B=C^0(B/A)$. The $i$-th coface map, denoted by $\phi_{1\cdots \hat{i}\cdots q}$, is induced by smashing with $A\rightarrow B$ after the first $i$-copies of $B$ and the $j$-th codegeneracy map is induced by smashing with $B\wedge_A B \rightarrow B$ after the first $j$ copies of $B$.
\end{definition}

\begin{definition}\label{descentdata}
Given a map $f:A\to B$ of $E_{\infty}$-rings, the $\infty$-category $\hbox{Tot}(QC(C^{\bullet}(B/A)))$ is the {\em category of descent data} for $f$.
\end{definition}

\begin{remark} The following observation justifies definition \ref{descentdata}. The objects of $\hbox{Tot}(QC(C^{\bullet}(B/A)))$ are in one-one correspondence with commutative diagrams of the following form:

$$\xymatrix{
\Delta^0 \ar[d] \ar[rr]^{d^0,d^1} &&\Delta^1 \ar[d] \ar[rr]^{d^0,d^1,d^2} \ar[d] &&\Delta^2 \ar[d]  &\cdots \\
B\hbox{-mod} \ar[rr]^{\phi_1^*,\phi_2^*} &&B^{\wedge_A^2}\hbox{-mod}\ar[rr]^{\phi_{12}^*,\phi_{23}^*,\phi_{13}^*} &&B^{\wedge_A^3}\hbox{-mod} &\cdots
}$$

We can think of an object of $\hbox{Tot}(QC(C^{\bullet}(B/A)))$ informally as the following data:

\begin{itemize}
\item A $B$-module $N$
\item An equivalence $\phi:\phi_1^*N \simeq \phi_2^*N$ in $B\wedge_AB$-mod
\item A $2$-simplex $\phi_{13}^*\phi \to \phi_{23}^*\phi \circ \phi_{12}^*\phi$ in $B\wedge_AB\wedge_AB$-mod
\item $\cdots$
\end{itemize}

The $k$-simplices of $\hbox{Tot}(QC(C^{\bullet}(B/A)))$ are in one-one correspondence with commutative diagrams 

$$\xymatrix{
\Delta^0\times \Delta^k \ar[d] \ar[rr]^{d^0,d^1} &&\Delta^1\times \Delta^k \ar[d] \ar[rr]^{d^0,d^1,d^2} \ar[d] &&\Delta^2 \times  \Delta^k \ar[d] &\cdots \\
B\hbox{-mod} \ar[rr]^{\phi_1^*,\phi_2^*} &&B^{\wedge_A^2}\hbox{-mod}\ar[rr]^{\phi_{12}^*,\phi_{23}^*,\phi_{13}^*} &&B^{\wedge_A^3}\hbox{-mod} &\cdots
}$$

\end{remark}

\begin{definition}\label{effdesc}
A map $f:A\to B$ of $E_{\infty}$-rings is of {\em effective descent for modules} when the map $$\theta_f:A\hbox{-mod} \to \hbox{Tot}(QC(C^{\bullet}(B/A)))$$ induced from the coaugmentation $A \to B = C^0(B/A)$ is an equivalence of stable $\infty$-categories.
\end{definition}

\subsection{Comonads in $\infty$-categories}\label{comonads}

In this section we discuss the theory of comonads in the $\infty$-categorical setting and the $\infty$-categorical analog of the Barr-Beck theorem as developed by Lurie. The main refererence for this is \cite[section 4]{HA}, in particular \cite[4.7]{HA}.  Lurie develops the theory of monads, we need the dual version of comonads here. We provide definitions of $\infty$-comonads and their $\infty$-categories of comodules.

\begin{definition}(see \cite[Def. 2.1.3.1, Def. 4.2.1.3]{HA})\label{algebra}
\begin{itemize}
\item Given a monoidal $\infty$-category $\mathcal{C}$, there is an {\em $\infty$-category of (associative) algebra objects in $\mathcal{C}$}  denoted by $\hbox{Alg}(\mathcal{C})$. 
\item Let $\mathcal{M}$ be an $\infty$-category left tensored over a monoidal $\infty$-category $\mathcal{C}$. Then there exists an {\em $\infty$-category of left module objects of $\mathcal{M}$} denoted by $\hbox{LMod}(\mathcal{M})$ and a map $$\hbox{LMod}(\mathcal{M}) \to \hbox{Alg}(\mathcal{C}).$$ If $A\in \hbox{Alg}(\mathcal{C})$, then we let $\hbox{LMod}_A(\mathcal{M})$ denote the fiber $\hbox{LMod}(\mathcal{M}) \times_{\hbox{Alg}(\mathcal{C})}\{A\}$. We refer to $\hbox{LMod}_A(\mathcal{M})$ as the {\em $\infty$-category of left $A$-modules in $\mathcal{M}$}.
\end{itemize}

\end{definition}

\begin{definition}(Coalgebras and comodules)
\begin{itemize}
\item Let $\mathcal{C}$ be a monoidal $\infty$-category. Define $\hbox{CoAlg}(\mathcal{C})$ to be $\hbox{Alg}(\mathcal{C}^{op})^{op}$. we refer to this as the {\em $\infty$-category  of (coassociative) coalgebra objects in $\mathcal{C}$}.

\item Let $\mathcal{M}$ be an $\infty$-category left tensored over a monoidal $\infty$-category $\mathcal{C}$. Define $\hbox{LComod}(\mathcal{M})$ to be $\hbox{LMod}(\mathcal{M}^{op})^{op}$. We refer to this as the {\em $\infty$-category  of (left) comodule objects of $\mathcal{M}$}. There is a map of $\infty$-categories $$\hbox{LComod}(\mathcal{M}) \to \hbox{CoAlg}(\mathcal{C}).$$ If $H \in \hbox{CoAlg}(\mathcal{C})$, then we let $\hbox{LComod}_H(\mathcal{M})$ denote the fiber $\hbox{LComod}(\mathcal{M}) \times_{\hbox{CoAlg}(\mathcal{C})}\{H\}$. We refer to $\hbox{LComod}_H(\mathcal{M})$ as the {\em $\infty$-category of left $H$-comodules in $\mathcal{M}$}.
Alternately, $\hbox{LComod}_H(\mathcal{M}) \simeq \hbox{LMod}_H(\mathcal{M}^{op})^{op}$.

\end{itemize}
\end{definition}

\begin{definition}(Comonads and comodules)
Given an $\infty$-category $\mathcal{D}$, the $\infty$-category of functors $\hbox{Fun}(\mathcal{D},\mathcal{D})$ is monoidal and $\mathcal{D}$ is left tensored over $\hbox{Fun}(\mathcal{D},\mathcal{D})$. 

\begin{itemize}
\item A functor $K \in \hbox{Fun}(\mathcal{D},\mathcal{D})$ is a {\em comonad} if $K \in \hbox{CoAlg}(\hbox{Fun}(\mathcal{D},\mathcal{D}))$.
\item There is an $\infty$-category $\hbox{LComod}_K(\mathcal{D})$ of comodules over a comonad $K$ in $\mathcal{D}$.
\end{itemize}
\end{definition}

There is a natural forgetful map $U_K:\hbox{LComod}_K(\mathcal{D}) \to \mathcal{D}$. 

\begin{remark} Informally, a comonad $K$ on an $\infty$-category $\mathcal{D}$ is an endofunctor $K:\mathcal{D}\to \mathcal{D}$ equipped with maps $K \to 1$ and $K \to K\circ K$ which satisfies the usual counit and co-associativity conditions up to coherent homotopy. A comodule over the comonad $K$ is an object $x\in \mathcal{D}$ equipped with a structure map $\eta: x \to K(x)$ which is compatible with the coalgebra structure on $K$, again up to coherent homotopy. The forgetful map takes a comodule to the underlying object in $\mathcal{D}$.\end{remark}

\begin{prop}(see \cite[Prop. 4.7.4.3]{HA})\label{composition}
Given a functor $F:\mathcal{C} \to \mathcal{D}$ of $\infty$-categories which admits a right adjoint $G$. Then the composition $K=F\circ G \in \hbox{Fun}(\mathcal{D},\mathcal{D})$ is a comonad on $\mathcal{D}$. 

There is a canonial map $F':\mathcal{C} \to \hbox{LComod}_K(\mathcal{D})$ so that $F'\circ U_K \simeq F \in \hbox{Fun}(\mathcal{C},\mathcal{D})$.
\end{prop}

\begin{remark} In ordinary categorical setting it is easy to check that the composition $K$ is a comonad on $\mathcal{D}$. However, as Lurie notes in \cite[Remark 4.7.0.4]{HA}, this a not so straightforward in the $\infty$-categorical setting. In order to give a coalgebra structure on the composition $K = F\circ G \in \hbox{Fun}(\mathcal{D}, \mathcal{D})$ it is not enough to give a produce a single natural transformation $K \to K\circ K$ but an infinite system of coherence data, which is not easy to describe explicitly.
\end{remark}

\begin{definition} Let $F:\mathcal{C} \to \mathcal{D}$ be a map of $\infty$-categories that admits a right adjoint $G$ and let $K=F\circ G$ be the composition comonad on $\mathcal{D}$. Then $F$ is said to be {\em comonadic} if the comparison map $F':\mathcal{C} \to \hbox{LComod}_K(\mathcal{D})$ is an equivalence of $\infty$-categories. 
\end{definition}

\begin{theorem}(see \cite[Thm.4.7.4.5]{HA})($\infty$-categorical Barr-Beck theorem)\label{bbl}
Let $F:\mathcal{C}\to \mathcal{D}$ be an $\infty$-functor that admits a right adjoint $G$. Then $F$ is comonadic if and only if $F$ satisfies the following two conditions:
\begin{itemize}
\item[(a)] $F$ reflects equivalences
\item[(b)] Let $U$ be a cosimplicial object in $\mathcal{C}$ which is $F$-split then $U$ admits a limit in $\mathcal{C}$ and the limit is preserved by $F$
\end{itemize}
\end{theorem}

\subsection{Higher descent data as comodules over an $\infty$-comonad}\label{hdd}

Our aim is to give an $\infty$-categorical version of Prop.\ref{comonadicdescent}. In other words, given a $E_{\infty}$-ring map $f:A\to B$ we want to express the category of descent data in Def.\ref{descentdata} as the category of comodules over a comonad on $B$-mod. (See \cite{Beardsley} for a different perspective on this.)

 Given a map $f:A \to B$ of $E_{\infty}$ rings, the map $f^*:A\hbox{-mod} \to B\hbox{-mod}$ of stable $\infty$-categories admits a right adjoint $f_*$. The compostion $K=f^*\circ f_*$, by Prop.\ref{composition}, is a comonad on $B$-mod. 

\begin{theorem}\label{hdd=comod}
There is an equivalence of stable $\infty$-categories, $$\hbox{Tot}(QC(C^{\bullet}(B/A))) \simeq \hbox{LComod}_K(B\hbox{-mod}).$$ 
\end{theorem}

First we need a few results from Lurie (\cite[Section 4.7.6]{HA}). 

\begin{prop}(see \cite[Prop.4.7.6.1]{HA})\label{preBC}
Let $\mathcal{C}^{\bullet}$ be a cosimplicial $\infty$-category and $F:\varprojlim \mathcal{C}^{\bullet} \to \mathcal{C}^0$ be the natural projection map. If $F$ admits a right adjoint then $F$ satifies the conditions of Thm. \ref{bbl}. i.e. there is a comonad $K$ on $\mathcal{C}^0$ so that there is an equivalence $\varprojlim\mathcal{C}^{\bullet} \simeq \hbox{LComod}_K(\mathcal{C}^0)$.
\end{prop}

\begin{definition}(Right-adjointability)(\cite[Def.4.7.5.13]{HA})\label{rightadjointable}
Given a diagram of $\infty$-categories $\sigma$:

$$\xymatrix{
\mathcal{C} \ar[r]^G \ar[d]_U &\mathcal{D} \ar[d]^V\\
\mathcal{C}' \ar[r] ^{G'}  &\mathcal{D}' 
}$$
and a specified equivalence $\alpha: G'\circ U \simeq V\circ G$. We say $\sigma$ is {\em right adjointable} if $G$ and $G'$ admit right adjoints $H$ and $H'$, and the composition transformation $$U\circ H \to H\circ G' \circ U \circ H \to^{\alpha} H'\circ V \circ G \circ H \to H'\circ V$$ is an equivalence.
\end{definition}

\begin{theorem}(\cite[Thm.4.7.6.2]{HA})\label{BC}
Let $\mathcal{C}^{\bullet}$ be a cosimplicial $\infty$-category. If for every $[m] \to [n]$ in $\Delta$ the diagram 
$$
\xymatrix{
\mathcal{C}^m \ar[r]^{d^0} \ar[d] &\mathcal{C}^{m+1} \ar[d] \\
\mathcal{C}^n \ar[r]^{d^0} \ar[r] &\mathcal{C}^{n+1} \\
}$$
is right adjointable (in particular, $d^0:\mathcal{C}^n \to \mathcal{C}^{n+1}$ admits a right adjoint $H(0)$), then 
\begin{enumerate}
\item the forgetful functor $\varprojlim\mathcal{C}^{\bullet} \to \mathcal{C}^0$ admits a right adjoint
\item The square $$\xymatrix{\varprojlim(\mathcal{C}^{\bullet}) \ar[d]_U \ar[r]^U &\mathcal{C}^0 \ar[d]^{d^1}\\ \mathcal{C}^0 \ar[r]^{d^0} &\mathcal{C}^1}$$ is right adjointable and there is an equivalence $U\circ H \simeq H(0) \circ d^1 \in \hbox{Fun}(\mathcal{C}^0,\mathcal{C}^0)$.
\end{enumerate}
\end{theorem}

\begin{remark} The the right adjointabily of the square in consequence (2) of Thm.\ref{BC} above  condition is similar to (\ref{BC1}).\end{remark}

\begin{proof}(of Thm.\ref{hdd=comod})
We need to show the projection map $p:\hbox{Tot}(QC(C^{\bullet}(B/A))) \to B\hbox{-mod}$ admits a right adjoint $p_!$ so that $p\circ p_! \simeq K$. Then the theorem can be proved by applying Prop.\ref{preBC}.

By Thm.\ref{BC}, $p$ admits a right adjoint if the following diagrams are right adjointable:
$$\xymatrix{
B^{\wedge_A^m}\hbox{-mod} \ar[r]^{d^0} \ar[d] &B^{\wedge_A^{m+1}}\hbox{-mod} \ar[d]\\
B^{\wedge_A^n}\hbox{-mod} \ar[r]^{d^0} \ar[r]^{d^0} &B^{\wedge_A^{n+1}}\hbox{-mod}\\
}$$

where coface maps $d^0$ are basechange along the map $1\wedge\cdots\wedge B :B^{\wedge_A^q} \to B^{\wedge_A^q}$. The maps $d^0$ admits right adjoints; forgetting the $B^{\wedge_A^q}$-module structure, and the right adjointability condition is satisfied. 

Therefore, by Thm.\ref{BC}(1), the map $p$ has a right adjoint, and by Thm.\ref{BC}(2), if $p_!$ is the right adjoint then $p\circ p_! \simeq H(0)\circ d^1$ as comonads over $B$-mod. It only remains to check that $H(0) \circ d_1 \simeq K = f^*\circ f_*$. This can be verified on objects, since for any $B$-module $N$, $(H(0)\circ d_1)(N) \simeq N\wedge_AB$ as a $B$-module.
\end{proof}

\subsection{Faithfully dualizable descent}\label{fdd}

\begin{definition}
Let $A$ and $B$ be $E_{\infty}$-rings.

\begin{enumerate}
\item An extension $A \rightarrow B$ is {\it faithful} if for any $A$-module $M$, $B\wedge_AM \simeq *$ implies $M\simeq *$.

\item Let $M$ be a $A$-module and $D_AM = F_A(M,A)$ the functional dual of $M$. Then $M$ is {\it dualizable} over $A$ if the canonical map $\nu:D_AM \wedge_AM \to F_A(M,M)$ is an equivalence.
\end{enumerate}
\end{definition}

\begin{lemma}(\cite[Lemma 3.3.2]{Rognes})\label{dual}
\begin{enumerate}
\item If $M$ is a dualizable over $A$ and $N$ any $A$-module then the canonical map $D_AM\wedge_AN \to F_A(M,N)$ is an equivalence.
\item If $M$ is dualizable over $A$, then $D_AM$ is also dualizable and the canonical map $M \to D_AD_AM$ is an equivalence. 
\end{enumerate}
\end{lemma}

\begin{prop}\label{fdimpliescomonadic}
Let $f:A \to B$ be a faithful map of $E_{\infty}$-rings and $B$ be dualizable over $A$ then the functor $$f^*:A\hbox{-mod} \to B\hbox{-mod}$$ is comonadic.
\end{prop}

\begin{proof}
Condition (i) of Thm.\ref{bbl} is satisfied since $A\to B$ is faithful. As for condition (ii), suppose $U^{\bullet}$ is a cosimplicial object in $A\hbox{-mod}$, then $\hbox{Tot}\,U^{\bullet} \in A\hbox{-mod}$. To show that $f^*$ preserves totalization we need $B$ is dualizable over $A$. To see this we note the following.

\begin{equation*}
\begin{split}
f^*\hbox{Tot}(U^{\bullet}) &= B\wedge_A\hbox{Tot}\,U^{\bullet}\\
 &\simeq F_A(D_AB, \hbox{Tot}\,U^{\bullet})\\
&\simeq \hbox{Tot}\,F_A(D_AB,U^{\bullet})\\
&= \hbox{Tot}(B\wedge_AU^{\bullet})
\end{split}
\end{equation*} 

The equivalences follow from Lemma \ref{dual}.

\end{proof}

As a corollary we get faithfully dualizable descent for modules:

\begin{theorem}\label{amit}
Let $f:A\to B$ be a faithful map of $E_{\infty}$-rings and $B$ be dualizable over $A$. Then $A\to B$ is of effective descent for modules.
\end{theorem}
\begin{proof}
Prop.\ref{fdimpliescomonadic} and Thm.\ref{hdd=comod}.
\end{proof}

\subsection{Galois extensions of $E_{\infty}$-ring spectra}

\begin{definition}\label{grCB} Let $\mathcal{C}$ be an $\infty$-category and let $X$ an object in $\mathcal{C}$ with an action by a finite group $G$, given by a map $X:BG \to \mathcal{C}$ of $\infty$-categories. Then the associated {\em group cobar complex} is a cosimplicial object $$C^{\bullet}(G;X): N(\Delta) \to \mathcal{C}$$ where $C^q(G;X)= \Pi_{G^q}X$ and the coface and codegeneracies are induced by the structure maps of $BG$. The limit $\hbox{Tot}(C^{\bullet}(G;X)) = X^{hG}$. 
\end{definition}

\begin{remark} If $B$ is a commutative $A$-algebra for an $E_{\infty}$-ring $A$ and $G$ acts on $B$ by $A$-algebra maps, then the cobar complex of Def.\ref{grCB} is a cosimplicial object $$C^{\bullet}(G;B): N(\Delta) \to \hbox{CAlg}(A\hbox{-mod})$$ and is coaugmented by $A\to B = C^0(B;G)$. This induces a map $A \to \hbox{Tot}(C^{\bullet}(B;G)) = B^{hG}$.
\end{remark} 

\begin{definition}(Rognes \cite{Rognes}) Let $f:A\to B$ be a map of $E_{\infty}$-rings and $G$ a finite group acting on $B$ through $A$-algebra maps. Then $f$ is a $G$-Galois extension if the canonical maps 
\begin{itemize}
\item[(i)] $i: A\rightarrow B^{hG}$ 
\item[(ii)] $h:B\wedge_A B \rightarrow \Pi_G B$ (informally given by $(b_1\wedge b_2) \mapsto \{b_1\wedge g(b_2)\}_{g\in G}$)
\end{itemize}
are equivalences. 
\end{definition}

\begin{definition}
Let $A\to B$ be a map of $E_{\infty}$-rings and let $G$ act on $B$ by $A$-algebra maps. Then there is a map of cosimplicial $A$-algebras $$h^{\bullet}:C^{\bullet}(B/A) \to C^{\bullet}(G;B)$$ given in codegree $q$ by the map $h^{q}:B^{\wedge_A^{q+1}} \to \Pi_{G^q}B$ given symbolically by $$b_0\wedge \ldots \wedge b_q \mapsto \{b_0\wedge g_1(b_1) \wedge \ldots \wedge g_q(b_q)\}_{(g_1,\ldots,g_q)\in G^q}.$$
\end{definition}

\begin{lemma}(\cite[Lemma 8.2.7]{Rognes})\label{galoisamit}
If $A\to B$ is a $G$-Galois extension then the map $h^{\bullet}$ is a codegreewise equivalence. 
\end{lemma}

\begin{prop}(\cite[6.2.1]{Rognes})\label{dualizable}
If $f:A\to B$ is a Galois extension then $B$ is dualizable over $A$.
\end{prop}

\subsection{Higher Galois descent}\label{galois}

Let $f:A\to B$ be a Galois extension of ordinary commutative rings. The truncated semi-cosimplicial ring of (\ref{CB}) is isomorphic to the following truncated semi-cosimplicial ring,

$$\xymatrix{
B \ar@<1ex>[r]^{\phi_1} \ar@<-1ex>[r]^{\phi_2} &\Pi_GB \ar[r]^{\phi_{12}} \ar@<2ex>[r]^{\phi_{23}} \ar@<-2ex>[r]^{\phi_{13}} &\Pi_{G\times G}B\\
}$$

where $\phi_1(x) = \left(g \mapsto gx \right)_{g \in G}$ and $\phi_2(x) = \left(g\mapsto x\right)_{g \in G}$. The category of descent data $$Desc(f) = \varprojlim \left( \xymatrix{
\hbox{Mod}(B) \ar@<1ex>[r]^{\phi_1^*} \ar@<-1ex>[r]^{\phi_2^*} &\Pi_G\hbox{Mod}(B) \ar[r]^{\phi_{12}^*} \ar@<2ex>[r]^{\phi_{23}^*} \ar@<-2ex>[r]^{\phi_{13}^*} &\Pi_{G\times G}\hbox{Mod}(B)\\
}\right)$$ can be denoted by $\hbox{Mod}(B)^G$. Given a $B$ module $N$, a descent datum is equivalent to a semi-linear $G$-action on $N$. In this setting, classical Galois descent can be stated as follows.

\begin{theorem}(Galois descent) Let $A\to B$ be a $G$-Galois extension of commutative rings, then there is an equivalence of categories $$\hbox{Mod}(A) \simeq \hbox{Mod}(B)^G.$$
\end{theorem}

\begin{definition}\label{fixedcategory}
Let $\mathcal{C}$ be an $\infty$-category with an action of a group $G$, given by a map $BG \to \mathcal{C}\hbox{at}_{\infty}$. The group cobar complex is a cosimplicial $\infty$-category $C^{\bullet}(G;\mathcal{C})$. We can consider the limit $$\mathcal{C}^{hG} = \hbox{Tot}(C^{\bullet}(G;\mathcal{C})).$$ An object $X$ of $\mathcal{C}$ will be called a {\em $G$-equivariant object of $\mathcal{C}$} if $X$ is an object of $\mathcal{C}^{hG}$.
\end{definition}

\begin{remark}
Informally, the objects of $\mathcal{C}^{hG}$ consist of the following data:
\begin{itemize}
\item An object $X \in \mathcal{C}$
\item An equivalence $\phi_{g,X}:g.X \to X$ for all $g\in G$
\item A $2$-simplex $\xymatrix{&X\\ g_1.X \ar[ur]^{\phi_{g_1,X}} &&g_2g_1.X \ar[ul]_{\phi_{g_1g_2,X}} \ar[ll]^{g_2,g_1.X}}$ for all $(g_1,g_2) \in G^2$
\item $\cdots$
\end{itemize}
\end{remark}

\begin{definition}\label{semilinear}
Let $B$ be an $E_{\infty}$-ring with an action of a group $G$. Then the $\infty$-category $B$-mod has an action of $G$. On objects it is given as follows. Let $N$ be a $B$-module with structure map $\alpha:B\wedge N \to N$. Then define $g.N$ to be the spectrum $N$ with a $B$-module structure by the map $\xymatrix{B\wedge N \ar[r]^{g\wedge 1} &B\wedge N \ar[r]^{\alpha} &N}$. We then refer to objects of the stable $\infty$-category $(B\hbox{-mod})^{hG}$ as {\em $G$-equivariant $B$-modules}.
\end{definition}

\begin{prop}\label{galoisdescentdata}Let $f:A\to B$ be  $G$-Galois extension of $E_{\infty}$-rings. Then there is an equivalence of stable $\infty$-categories $$Desc(f) \simeq (B\hbox{-mod})^{hG}.$$
\end{prop}
\begin{proof}
There is an equivalence $$\hbox{Tot}(QC(C^{\bullet}(G;B)) \simeq \hbox{Tot}(C^{\bullet}(G;B\hbox{-mod}))$$ since there is a codegreewise equivalance $(\Pi_{G^q}B)\hbox{-mod} \simeq \Pi_{G^q}(B\hbox{-mod})$. Also, by Lemma \ref{galoisamit} the map $$\hbox{Tot}(QC(h^{\bullet})): \hbox{Tot}(QC(C^{\bullet}(B/A)) \to \hbox{Tot}(QC(C^{\bullet}(G;B)))$$ induced by $h^{\bullet}$ is an equivalence of stable $\infty$-categories.
\end{proof}

We can state the Galois descent theorem for a faithful Galois extension of $E_{\infty}$-rings.

\begin{theorem}(Galois descent for $E_{\infty}$-rings)\label{GaloisDesc}
Let $f:A\to B$ be a faithful $G$-Galois extension of $E_{\infty}$-rings. Then the map  $$ A\hbox{-mod} \to (B\hbox{-mod})^{hG}$$ given by $N \mapsto  \hbox{Tot}(QC(h^{\bullet}))(\theta_f(N))$ is an equivalence of stable $\infty$-categories of the category of $A$-modules and the catgeory of $G$-equivariant $B$-modules.
\end{theorem}
\begin{proof}
Apply Thm.\ref{amit}, Prop. \ref{dualizable} and Prop.\ref{galoisdescentdata}.

\end{proof}

\section{Generalized Real Johnson-Wilson theories}

\subsection{$\mathbb{Z}/2$-equivariant spectra}\label{Z2}

Let $\alpha$ denote the one-dimesional sign representation of $\mathbb{Z}/2$ and $\mathcal{U} = \mathbb{R}^{\infty} \oplus \mathbb{R}^{\infty\alpha}$ be a complete $\mathbb{Z}/2$-universe. An indexing space is a finite-dimensional subrepresenatation of $\mathcal{U}$, therefore of the form $V=m+n\alpha$. If $V\subset W$, let $W-V$ be the orthogonal complement of $V$ in $W$.

A $\mathbb{Z}/2$-spectrum $X_{\mathbb{Z}/2}$ is a collection of pointed $\mathbb{Z}/2$-spaces $X_{\mathbb{Z}/2}(V)$ and a system of $\mathbb{Z}/2$-homeomorphisms $$X_{\mathbb{Z}/2}(V) \simeq \Omega^{W-V}X_{\mathbb{Z}/2}(V).$$ 

Since any indexing space is contained in $n(1+\alpha)$ for a large enough $n$, we have the following simple description of a $\mathbb{Z}/2$-spectrum.

\begin{lemma}
A $\mathbb{Z}/2$-spectrum $X_{\mathbb{Z}/2}$ is a collection of pointed $\mathbb{Z}/2$-equivariant spaces $X_n$ with equivariant structure maps, $$S^{1+\alpha}\wedge X_n \to X_{n+1}.$$
\end{lemma}

\begin{definition}(\cite{GM})
The Borel, co-Borel, geometric and Tate spectra associated with $X_{\mathbb{Z}/2}$ are defined respectively as follows:
\begin{itemize}
\item[] $c({X_{\mathbb{Z}/2}}) = F(E{\mathbb{Z}/2}_+,{X_{\mathbb{Z}/2}})$ 
\item[] $f(X_{\mathbb{Z}/2}) = E{\mathbb{Z}/2}_+\wedge X_{\mathbb{Z}/2}$
\item[] $g({X_{\mathbb{Z}/2}}) = {X_{\mathbb{Z}/2}} \wedge \tilde{E}{\mathbb{Z}/2}$ 
\item[] $t({X_{\mathbb{Z}/2}}) = F(E{\mathbb{Z}/2}_+,{X_{\mathbb{Z}/2}}) \wedge \tilde{E}{\mathbb{Z}/2}$ 
\end{itemize}
where $\tilde{E}\mathbb{Z}/2$ is the unreduced suspension of $E\mathbb{Z}/2$.
\end{definition}

The Tate diagram is a commutative diagram of ${\mathbb{Z}/2}$-spectra.
\begin{equation}
\xymatrix{
f(X_{\mathbb{Z}/2}) \ar[r] \ar[dr] &X_{\mathbb{Z}/2} \ar[r] \ar[d] &g(X_{\mathbb{Z}/2}) \ar[d]\\
&c(X_{\mathbb{Z}/2}) \ar[r] &t(X_{\mathbb{Z}/2})\\
}
\end{equation}
The top and bottom rows are fibrations of $\mathbb{Z}/2$-spectra.

There are spectral sequences (Borel, co-Borel and Tate resp.)
\begin{equation}
\begin{split}
H^p({\mathbb{Z}/2};{X_{\mathbb{Z}/2}}^q) \Rightarrow c(X_{\mathbb{Z}/2})_{\star}\\
H_p({\mathbb{Z}/2};{X_{\mathbb{Z}/2}}^q) \Rightarrow f(X_{\mathbb{Z}/2})_{\star}\\
\widehat{H}^p({\mathbb{Z}/2};{X_{\mathbb{Z}/2}}^q) \Rightarrow t(X_{\mathbb{Z}/2})_{\star}\\
\end{split}
\end{equation}
where $q \in RO(\mathbb{Z}/2)$ and $p \geq 0, p\geq 0$ and $p \in \mathbb{Z}$ respectively.

The {\it homotopy fixed points} $X^{h{\mathbb{Z}/2}}$ of $X_{\mathbb{Z}/2}$ is the ordinary fixed points of the Borel spectrum. The standard inclusion map ${\mathbb{Z}/2}_+ \subset E{\mathbb{Z}/2}_+$ induces a map of non-equivariant spectra. 
$$X^{h\mathbb{Z}/2} \to X = F({\mathbb{Z}/2}_+,X_{\mathbb{Z}/2})^{\mathbb{Z}/2}$$

\subsection{Real-oriented cohomology theories}\label{real}

In this section we recall the construction of real-oriented spectra from \cite{HK}. Let $MU(n)$ denote the Thom space of the universal bundle $\gamma_n$ over $BU(n)$. Complex conjugation induces an action of $\mathbb{Z}/2$ on $MU(n)$. The canonical Real bundle $\gamma_n$ of dimension $n$ over $BU(n)$ gives equivariant maps between Thom spaces, $\Sigma^{1+\alpha}BU(n)^{\gamma_n} \rightarrow BU(n+1)^{\gamma_{n+1}}$. The resulting genuine $\mathbb{Z}/2$-spectrum is denoted $MU_{\mathbb{R}}$. The spectrum $MU_{\mathbb{R}}$ is an $E_{\infty}$ ring spctrum. The underlying non-equivariant spectrum of ${MU_{\mathbb{R}}}$ is $MU$.

\begin{definition}
Let $B\mathbb{S}^1$ be the classifying space of $\mathbb{S}^1$ considered as $\mathbb{Z}/2$ equivariant space via the inclusion $\mathbb{S}^1 \subset \mathbb{C}^*$. We have $\Omega B \mathbb{S}^1 \simeq \mathbb{S}^1$ in the category of based $\mathbb{Z}/2$-spaces. Therefore by adjunction we have a canonical equivariant based map $\eta: S^{1+\alpha} \rightarrow B\mathbb{S}^1$. Let $X_{\mathbb{Z}/2}$ be a homotopy commutative and associative ${\mathbb{Z}/2}$-ring spectrum. A {\it Real orientation} of $X_{\mathbb{Z}/2}$ is a cohomology class $u:B\mathbb{S}^1 \rightarrow \Sigma^{1+\alpha}X_{\mathbb{Z}/2}$ such that $\eta^*u = 1$ in $\pi_0X_{\mathbb{Z}/2}$.
\end{definition}

The spectrum $MU_{\mathbb{R}}$ is real-oriented, hence it supports a formal group law. The forgetful map ${MU_{\mathbb{R}}}_{\star} \to MU_{\star}$ is split by the map of rings $MU_* \to {MU_{\mathbb{R}}}_{\star}$ classifying this formal group law, where the image of $x_i \in MU_{2i}$ is in degree $i(1+\alpha)$. Here $\star$ denotes the bigraded coefficient ring of any $\mathbb{Z}/2$-spectrum.

\subsection{$BP\langle n, \bold{u} \rangle_{\mathbb{R}}$ and $E(n;\bold{u})_{\mathbb{R}}$}\label{generalized}

Working $2$-locally there is a Real analogue of the Quillen idempotent which produces a $\mathbb{Z}/2$-equivariant spectrum $BP_{\mathbb{R}}$ such that $\left(MU_{\mathbb{R}}\right)_{(2)}$ splits as a wedge of suspensions of $BP_{\mathbb{R}}$ (\cite[Thm.2.33]{HK}).

Using the map $MU_* \to {MU_{\mathbb{R}}}_{\star}$ we may identify classes $v_n \in \pi_{(2^n-1)(1+\alpha)}MU_{\mathbb{R}}$. Using these elements some constructions of complex-oriented spectra can be mimicked to give real versions. Consider the $MU$-module spectra $E(n)$ and $BP\langle n \rangle$ with coefficient rings:
$$BP\langle n \rangle_* = \mathbb{Z}_{(2)}[v_1,\ldots,v_{n}]$$
$$E(n)_*  = \mathbb{Z}_{(2)}[v_1,\ldots,v_{n-1},v_n^{\pm 1}]$$ 

We can mod out by the lifts $v_{n+1}, v_{n+2}, \cdots \in \pi_{\star}BP_{\mathbb{R}}$ to construct Real truncated $BP$ theory $BP\langle n \rangle _{\mathbb{R}}$. We can also invert $v_n$ to construct Real Johnson Wilson theory $E(n)_{\mathbb{R}}$ (see \cite[Sec.3]{HK}).

\begin{definition}
Let $\bold{u}=u_0,u_1,\ldots,u_n,\ldots$ be a regular sequence in $BP_*$ such that $|u_i| = 2(2^i-1)$ and $(2,u_1,u_2,\ldots,u_{n-1})=I_n$. There are commutative ring spectra (see \cite{Baker}) $BP\langle n ; \bold{u}\rangle$ and $E(n;\bold{u})$ such that $$BP\langle n; \bold{u} \rangle_* = BP_*/(u_i: i \geq n+1)$$ and $$E(n;\bold{u})_* = BP_*/(u_i: i\geq n+1)[u_n^{-1}].$$ We refer to $E(n;\bold{u})$ as a {\em genralized Johnson-Wilson spectrum}.
\end{definition}

\begin{remark}
 If we set $\bold{u}=\bold{v}$, the Hazewinkel generators of $BP_*$, then we recover the standard Johnson-Wilson spectra $BP\langle n \rangle$ and $E(n)$. The rings $\pi_*E(n;\bold{u})$ are isomorphic for different choices of $\bold{u}$, but support different formal group laws and therefore are non-equivalent as complex-oriented spectra.
\end{remark}

\begin{definition}
We can identify classes $u_i \in \pi_{(2^i-1)(1+\alpha)}BP_{\mathbb{R}}$ which are lifts of $u_i \in BP_{(2^i-1)2}$. Going modulo the appropriate classes in $\pi_{\star}BP_{\mathbb{R}}$ we can define real-oriented spectra $BP\langle n; \bold{u}\rangle_{\mathbb{R}}$ and $E(n;\bold{u})_{\mathbb{R}}$. We refer to $E(n;\bold{u})_{\mathbb{R}}$ as a {\em generalized Real Johnson-Wilson sprectrum}.
\end{definition}

\subsection{Fibrations related to $E(n;\bold{u})_{\mathbb{R}}$}\label{fib}
 
\begin{theorem}\label{comp}Let $X$ be a generalized Johnson-Wilson spectrum $E(n;\bold{u})$ and $X_{\mathbb{R}}=E(n;\bold{u})_{\mathbb{R}}$ the associated real-oriented spectrum.

The following are true.
\begin{itemize}
\item[(i)] There is an invertible element $y(n) \in {c(X_{\mathbb{R}})}_{\star}$ in degree $\lambda(n)+\alpha$, where $\lambda(n) = 2^{2n+1}-2^{n+2}+1$.
\item[(ii)] The Euler class $a \in {c(X_{\mathbb{R}})}_{\star}$ in degree $\alpha$, coming from $a:S^0\to S^{\alpha}$, is nilpotent.

\item[(iii)] The Tate spectrum $t(X_{\mathbb{R}})$ is trivial.
\end{itemize}
\end{theorem}

\begin{proof} Using Cor.\ref{BERnu} define the invertible element $y(n) = u_n^{2^n-1}\sigma^{-2^{n+1}(2^{n-1}-1)}\in c(X_{\mathbb{R}})_{\star}$. Since $u_n$ is invertible in $c(X_{\mathbb{R}})_{\star}$ it follows from the differential of equation (\ref{diffBSS}) that $a^{2^{n+1}-1}=0$ in $c(X_{\mathbb{R}})_{\star}$.

Recall that $t(X_{\mathbb{Z}/2})$ is the localization of $c(X_{\mathbb{Z}/2})$ away from $a$ (\cite[16.3]{GM}). In the case of $X_{\mathbb{R}}$ there is a commutative square of $\mathbb{Z}/2$-equivariant ring spectra:

$$\xymatrix{
X_{\mathbb{R}} \ar[d] \ar[r] &g(X_{\mathbb{R}}) \ar[d] \ar@{}[r] &=X_{\mathbb{R}}[a^{-1}] \\
c(X_{\mathbb{R}}) \ar[r] &t(X_{\mathbb{R}}) \ar@{}[r] &=c(X_{\mathbb{R}})[a^{-1}]\\
}$$

The element $a:S^0 \subset S^{\alpha}$ acts
\begin{itemize}
\item[(a)] nilpotently on $c(X_{\mathbb{R}})_{\star}$.
\item[(b)] invertibly on $g(X_{\mathbb{R}})_{\star}$ and $t(X_{\mathbb{R}})_{\star}$.
\end{itemize}
 
The Tate spectrum $t(X_{\mathbb{R}})$ is the localization of $c(X_{\mathbb{R}})$ away from $a$, so on $t(X_{\mathbb{R}})_{\star}$, $a$ acts invertibly as well as nilpotently. Therefore $t(X_{\mathbb{R}})$ is equivariantly contractible.
\end{proof}

\begin{theorem}\label{KWfib}
There is a fibration of $XR$-algebras,
\begin{equation}\label{cof}
\xymatrix{
\Sigma^{\lambda(n)}XR \ar[r]^{x(n)} &XR \ar[r] &X\\
}
\end{equation}
where $XR := X_{\mathbb{R}}^{h\mathbb{Z}/2}$ The element $x(n)$ has degree $\lambda(n) = 2^{2n+1} -2^{n+2}+1$ and is nilpotent with $x(n)^{2^{n+1}-1}=0$. 
\end{theorem}

\begin{example}-
\begin{itemize}
\item[(i)]
When $X=KU$, this is the fibration $$\xymatrix{\Sigma^1KO \ar[r]^{\eta} &KO \ar[r]^c &KU}$$ of $KO$-module spectra. The map $c$ is complexification and the fiber can be identified with $\Sigma KO$ using the equivalence $KU \simeq KO\wedge C_{\eta}$ of Reg Wood. Here $\eta$ is the stable Hopf map $\eta:S^1\rightarrow S^0$ and $\eta^3=0$.

\item[(ii)]
The spectrum of $TMF_1(3)$ of topological modular forms with $\Gamma_1(3)$ level structures is an example of $X$ when $n=2$. There is a $\mathbb{Z}/2= \Gamma_0(3)/\Gamma_1(3)$ action on $TMF_1(3)$ coming from the level $3$ structures. In a recent preprint (\cite{HM}) Hill and Meier has shown that $TMF_1(3)$ with this $\mathbb{Z}/2$-action is real-oriented, so that the $\mathbb{Z}/2$-homotopy fixed points of $TMF_1(3)_{\mathbb{R}}$ is $TMF_0(3)$. The fibration of Theorem \ref{KWfib} then is the Maholwald-Rezk fibration (\cite[Remark 4.2]{MR}),
$$\xymatrix{\Sigma^{17}TMF_0(3) \ar[r]^{x} &TMF_0(3) \ar[r] &TMF_1(3)}$$ of $TMF_0(3)$-module spectra. The element $x \in \Sigma^{17}TMF_0(3)$ is nilpotent with $x^7=0$ (see \cite[Prop 4.1]{MR}).

\item[(iii)] More generally, with $X=E(n)$ the $n$-th Johnson-Wilson spectrum, Kitchloo and Wilson (\cite{KW1}) have produced fibrations $$\Sigma^{\lambda(n)}ER(n) \to ER(n) \to E(n).$$

\end{itemize}
\end{example}

\begin{proof}
From Thm.\ref{comp} we have the following equivalence as part of the Tate diagram.

$$\xymatrix{
f(X_{\mathbb{R}}) \ar[r] \ar[dr]^{\simeq} &X_{\mathbb{R}} \ar[d]  \ar[r] &g(X_{\mathbb{R}})\\
&c(X_{\mathbb{R}})\\
}$$
This implies a splitting of $\mathbb{Z}/2$- ring spectra 
$$ X_{\mathbb{R}} \simeq c(X_{\mathbb{R}}) \vee g(X_{\mathbb{R}}).$$

There is a fibration
$\xymatrix{
\mathbb{Z}/2_+ \ar[r] &S^0 \ar[r]^a &S^{\alpha}\\
}$
inducing a fibration
$$\xymatrix{
F(S^{\alpha}, g(X_{\mathbb{R}})) \ar[d]^{\simeq} \ar[r]^{a^*} &F(S^0,g(X_{\mathbb{R}})) \ar[r] \ar[d]^{\simeq} &F(\mathbb{Z}/2_+, g(X_{\mathbb{R}}))\\
\Sigma^{-\alpha}g(X_{\mathbb{R}}) \ar[r]^a_{\simeq} &g(X_{\mathbb{R}})\\
}$$
This equivalence induced by $a$ implies that $F(\mathbb{Z}/2_+,g(X_{\mathbb{R}}))$ is trivial. We have the analogous fibration:

$$\xymatrix{
\Sigma^{-\alpha}c(X_{\mathbb{R}}) \ar[r]^a &c(X_{\mathbb{R}}) \ar[r] &F(\mathbb{Z}/2_+,c(X_{\mathbb{R}})) \ar[d]^=\\
\Sigma^{\lambda(n)}c(X_{\mathbb{R}}) \ar[u]^{y(n)}_{\simeq} \ar[ur]_{x(n)} &&F(\mathbb{Z}/2_+,X_{\mathbb{R}})\\
}$$

The desired fibration is obtained by taking the ordinary fixed points of the bottom fibration. 

\end{proof}

\begin{theorem}\label{faith} Let $X$ be a generalized Johnson-Wilson-spectrum having a $E_{\infty}$-ring structure. Then the extension $XR \rightarrow X$ is a faithful $\mathbb{Z}/2$-Galois extension.
\end{theorem}
\begin{proof}
The proof proceeds along the lines of \cite[Prop. 5.3.1]{Rognes}. The main ingredient is the fibration of Prop. \ref{KWfib}. By definition, $XR = c(X_{\mathbb{R}})^{\mathbb{Z}/2}$ and the group $\mathbb{Z}/2$ acts through $XR$-algebra maps. Therefore in order to show $XR\to X$ is Galois it only remains to show that the map $$h:X\wedge_{XR}X \rightarrow \Pi_{\mathbb{Z}/2}X$$ is an equivalence. Consider the part of the cofibration (\ref{cof}) induced by fibration 
\begin{equation}\label{cof2}
S^{\alpha -1} \rightarrow \mathbb{Z}/2_+ \rightarrow S^0.
\end{equation}

$$\xymatrix{
XR\ar@{=}[d] \ar[r] &X \ar@{=}[d] \ar[r] &\Sigma^{\lambda(n)+1}XR \ar@{=}[d] \\
c(X_{\mathbb{R}})^{\mathbb{Z}/2} \ar[r] &c(F(\mathbb{Z}/2_+,X_{\mathbb{R}}))^{\mathbb{Z}/2} \ar[r] &c(\Sigma^{\alpha-1}X_{\mathbb{R}})^{\mathbb{Z}/2}\\
}$$

The homotopy fixed points inclusion $$c({X_G})^G \rightarrow X$$ then produces a commutative diagram.

$$\xymatrix{
XR\ar[d] \ar[r] &X \ar[d] \ar[r] &\Sigma^{\lambda(n)+1}XR \ar[d] \\
X \ar[r]_{\Delta} &F(\mathbb{Z}/2_+,X) \ar[r]_{\delta} &X\\
}$$
This is a map of cofibrations and the bottom row is induced by the non-equivariant version of the cofibration (\ref{cof2}) on $X$. The map $\Delta$ is the trivial $\mathbb{Z}/2$-Galois extension over $X$ ($\Delta$ is the diagonal inclusion) and the homotopy cofiber $\delta$ can be identified with $X$.

This is a diagram of $XR$-algebras. Inducing the upper row along $XR \rightarrow X$, gives the following map between cofiber sequences. Furthermore, by adjuncton, this is a commutative diagram of $X$-algebras.

$$
\xymatrix{
X\wedge_{XR}XR \ar[r] \ar[d]_{\simeq} &X\wedge_{XR}X \ar[r] \ar[d] &X\wedge_{XR}\Sigma^{\lambda(n)+1}XR \ar[d]^{\simeq}\\
X \ar[r]_{\Delta} &\Pi_{\mathbb{Z}/2}X \ar[r]_{\delta} &X\\
}$$

The right hand column is an equivalence since $X\wedge_{XR}\Sigma^{\lambda(n)+1}XR \simeq \Sigma^{\lambda(n)+1}X \simeq X$. This follows from the fact that $X$ is $|v_n|=2(2^n-1)$ periodic and $2(2^n-1)(2^n-1) = \lambda(n)+1$. Therefore $h$ is an equivalence.

Finally, $XR \rightarrow X$ is faithful. If $N$ is a $XR$-module then applying to the cofibration (\ref{cof}) we get a cofibration 
$$\xymatrix{
\Sigma^{\lambda(n)}N \ar[r]^{x(n)} &N \ar[r] &N\wedge_{XR}X.
}$$
If we assume $N\wedge_{XR}X \simeq *$, then $x(n):\Sigma^{\lambda(n)}N \rightarrow N$ is an equivalence. But $x(n)$ is nilpotent. Therefore $N \simeq *$.

\end{proof}

\appendix 

\section{Borel spectral sequence for $E(n;\bold{u})_{\mathbb{R}}$}\label{comps}
 
In this section we compute the coefficient ring of the Borel spectrum of $E(n;\bold{u})_{\mathbb{R}}$. Hu an Kriz have computed the Borel spectral sequence for $BP_{\mathbb{R}}$ and $E(n)_{\mathbb{R}}$. Here we recall their results.

\begin{definition}\label{E_infty}
Let $\bold{u}=u_0,u_1,\ldots,u_n,\ldots$ be a regular sequence in $BP_*$ such that $|u_i| = 2(p^i-1)$ and $(p,u_1,u_2,\ldots,u_{n-1})=I_n$. Define the following ring.
$$E_{\infty}^c((BP;\bold{u})_{\mathbb{R}}) = \mathbb{Z}_{(2)}[u_n\sigma^{l2^{n+1}}, a | l \in \mathbb{Z}, n \geq 0]/I(\bold{u})$$

The ideal $I(\bold{u})$ is generated by the following relations,
\begin{equation}
\begin{split}
u_0 &=2\\
\left(u_n\sigma^{l2^{n+1}}\right) a^{2^{n+1}-1} &= 0\\
\left(u_m\sigma^{k2^{m+1}}\right) \left(u_n\sigma^{l2^{m-n}2^{n+1}}\right)  &= u_nu_m \sigma^{(k+l)2^{m+1}}\\
\end{split}
\end{equation}

where the $u_n\sigma^{l2^{n+1}}$ has bidegree $(2^n-1)(1+\alpha) + l2^{n+1}(\alpha - 1)$ and $a$ has bidegree $-\alpha$. 
\end{definition}

\begin{theorem}(Hu,Kriz \cite{HK})
The Borel spectral sequence for $BP_{\mathbb{R}}$ is $$H^*(\mathbb{Z}/2, BP_*[\sigma^{\pm}]) \Rightarrow \pi_{\star}c(BP_{\mathbb{R}})$$ where $\sigma \in \pi_{1-\alpha}F(\mathbb{Z}/2_+,BP_{\mathbb{R}})$ coming from the homeomorphism $S^1\wedge \mathbb{Z}/2_+ \simeq S^{\alpha}\wedge \mathbb{Z}/2_+$. The differentials are 
\begin{equation}
d_{2^{k+1}-1}(\sigma^{-2^k}) = v_ka^{2^{k+1}-1}
\end{equation}
where $a \in \pi_{\alpha}BP_{\mathbb{R}}$ comes from the embedding $S^0 \subset S^{\alpha}$.

The $E_{\infty}$-page is $E_{\infty}^c((BP;\bold{v})_{\mathbb{R}})$. Furthermore, there are no extension problems. Therefore, 

$$c(BP_{\mathbb{R}})_{\star} \simeq E_{\infty}^c((BP;\bold{v})_{\mathbb{R}}). $$

\end{theorem}

In the following proposition we show that running the Borel spectral sequence with different generators $\bold{u}$ for $BP_*$ gives similar results.

\begin{prop}\label{BP}
For any $\bold{u}$ in Def.\ref{E_infty}, 
$$c(BP_{\mathbb{R}})_{\star} \simeq E_{\infty}^c((BP;\bold{u})_{\mathbb{R}}).$$
\end{prop}

\begin{proof} 

We compute the Borel spectral sequence for $BP_{\mathbb{R}}$  starting with the generators $\bold{u}$ of $BP_*$.

Since $(2,u_1,\ldots, u_k) = (2,v_1,\ldots, v_k)$, $u_k=\alpha_kv_k \,\,\hbox{mod}\,I_n$ where $\alpha_k$ is a unit. In the $E_{2^{k+1}-1}$-page
\begin{equation}
\begin{split}
u_ka^{2^k-1} &= \left(\alpha_kv_k + \Sigma_0^{k-1}x_iv_i\right) a^{2^k-1}\\
&=\alpha_kv_ka^{2^k-1}\\
\end{split}
\end{equation}

since the other terms are zero. Therefore we can rewrite the differentials as 

\begin{equation} \label{diffBSS}
d_{2^{k+1}-1}(\sigma^{-2^k}) = v_ka^{2^{k}-1}a^{2^{k+1}-2^k} = \alpha_k^{-1}u_ka^{2^{k+1}-1}.
\end{equation}

Since the differentials don't change under the new generators (up to a unit), the $E_{\infty}$-page is $E_{\infty}^c((BP;\bold{u})_{\mathbb{R}})$.

Next, we have to show there are no extension problems. For this it is enough to show that $E_{\infty}^c((BP;\bold{u})_{\mathbb{R}})$ and $E_{\infty}^c((BP;\bold{v})_{\mathbb{R}})$ have the same multiplicative structure. 

We have to show $I(\bold{u}) = I(\bold{v})$. We have $u_0=v_0=2$. To show, $$0 = v_n \sigma^{l2^{n+1}}a^{2^{n+1}-1} = u_n\sigma^{l2^{n+1}}a^{2^{n+1}-1}.$$
Proceed by induction on $n$. For $n=1$, $u_1 = v_1 \,\,\hbox{mod}(2)$, let $u_1 = \alpha_1v_1 + 2x$, then $u_1\sigma^{l4}a^{3} = (\alpha_1v_1 + 2x)\sigma^{4l}a^{3} = \alpha_1v_1\sigma^{4l}a^{3} + x2\sigma^{4l}a^{3}$. But $v_0\sigma^{2l}a = 0 \Rightarrow x2\sigma^{4l}a^{3} = 0$. 

$u_n = \alpha_nv_n \,\, \hbox{mod}\,I_n \Rightarrow $
\begin{equation}
\begin{split}
u_n\sigma^{l2^{n+1}}a^{2^{n+1}-1} &= (\alpha_nv_n + \Sigma_0^{n-1} x_iv_i)\sigma^{l2^{n+1}}a^{2^{n+1}-1}\\
&= \alpha_nv_n\sigma^{l2^{n+1}}a^{2^{n+1}-1}\\
\end{split}
\end{equation}
by induction hypothesis. 

To show, $\left( u_m\sigma^{k2^{m+1}} \right) \left(u_n\sigma^{l2^{n+1}} \right) = u_nu_m \sigma^{k2^{m+1} + l2^{n+1}}$. Since $u_n = \alpha_nv_n \,\,\hbox{mod}\,I_n$ and $u_m = \alpha_mv_m \,\,\hbox{mod}\,I_m$, the left hand side 

\begin{equation}
\begin{split}
\left( u_m\sigma^{k2^{m+1}} \right) \left(u_n\sigma^{l2^{n+1}} \right)
&=\left( (\alpha_mv_m + \Sigma_0^{m-1}x_iv_i) \sigma^{k2^{m+1}} \right) \left( (\alpha_nv_n + \Sigma_0^{n-1}y_jv_j)\sigma^{l2^{n+1}} \right)\\
 &= \left( \alpha_mv_m \sigma^{k2^{m+1}} \right) \left( \alpha_nv_n \sigma^{l2^{n+1}} \right) + \Sigma_{0,0}^{m-1,n-1} \left(x_iv_i\sigma^{k2^{m+1}}\right)\left(y_jv_j\sigma^{l2^{n+1}}\right)\\
&= \alpha_m\alpha_nv_mv_n\sigma^{k2^{m+1}+l2^{n+1}} + \Sigma_{0,0}^{m-1,n-1} \left(x_iy_jv_iv_j\sigma^{k2^{m+1}+l2^{n+1}}\right)\\
&= \left( \alpha_m\alpha_nv_mv_n + \Sigma_{0,0}^{m-1,n-1}x_iy_jv_iv_j \right) \sigma^{k2^{m+1}+l2^{n+1}}\\
&= u_mu_n\sigma^{k2^{m+1}+l2^{n+1}}.\\
\end{split}
\end{equation}

\end{proof}

As a corollary we obtain the coefficient ring of the Borel spectrum of $E(n;\bold{u})_{\mathbb{R}}$.

\begin{cor}\label{BERnu} The $\mathbb{Z}/2$-equivariant homotopy of the generalized Real Johnson-Wilson spectrum is given by,
$$\pi_{\star}c(E(n;\bold{u})_{\mathbb{R}}) = \mathbb{Z}_{(2)}[u_k\sigma^{l2^{k+1}},a,u_n^{\pm 1}, \sigma^{\pm 2^{n+1}}]/I(\bold{u}), \,\,\, l\in \mathbb{Z}, n,k \geq 0.$$ 
\end{cor}

\end{document}